\documentclass[12pt]{amsart}
\usepackage{amsmath,amsthm,amsfonts,amssymb}
\date{\today}

\usepackage{color}
\input xymatrix
\xyoption{all}

\usepackage{hyperref}

\newcommand{\mathscr}{\mathcal}
\newcommand{\w}{\omega}

\newtheorem{theorem}{Theorem}%[section]

\newtheorem{lemma}[theorem]{Lemma}
\theoremstyle{definition}
\newtheorem*{mtheo*}{Main Theorem}
\newtheorem*{coro*}{Corollary}

\begin{document}

\title[Classifying locally compact semitopological polycyclic monoids]{Classifying locally compact semitopological polycyclic monoids}

\author[S.~Bardyla]{Serhii~Bardyla}
\address{Faculty of Mathematics, National University of Lviv,
Universytetska 1, Lviv, 79000, Ukraine}
\email{sbardyla@yahoo.com}



\begin{abstract}
We present a complete classification of Hausdorff locally compact polycyclic monoids up to a topological isomorphism. A {\em polycyclic monoid} is an inverse monoid with zero, generated by a subset $\Lambda$ such that $xx^{-1}=1$ for any $x\in\Lambda$ and $xy^{-1}=0$ for any distinct $x,y\in\Lambda$. We prove that any non-discrete Hausdorff locally compact topology with continuous shifts on a polycyclic monoid $M$ coincides with the topology of one-point compactification of the discrete space $M\setminus\{0\}$.
\end{abstract}

\maketitle

\section*{Introduction}
In this paper we present a complete classification of locally compact semitopological polycyclic monoids up to a topological isomorphism.

%All topological spaces considered in this paper are assumed to be Hausdorff.
We shall follow the terminology of~\cite{Clifford-Preston-1961-1967,
Engelking-1989, Lawson-1998, Ruppert-1984}. First we recall some information on inverse semigroups and monoids. %By $\mathbb{N}$ we denote the set of all positive integers and by $\w$ the set of all finite ordinals. So, $\w=\{0\}\cup\mathbb{N}$.
We identify cardinals with the sets of ordinals of smaller cardinality.

A {\em semigroup} is a set $S$ endowed with an associative binary operation $\cdot:S\times S\to S$, $\cdot:(x,y)\mapsto xy$. An element $e\in S$ is called the {\em unit} (resp. {\em zero}) of $S$ if $xe=x=ex$ (resp. $xe=e=ex$) for all $x\in S$. A semigroup can contains at most one unit (which will be denoted by $1$) and at most one zero (denoted by $0$).  A {\em monoid} if a semigroup with a unit.

A semigroup $S$ is called \emph{inverse} if for every element $a\in S$ there exists a unique element $a^{-1}$ (called the {\em inverse} of $a$) such that $aa^{-1}a=a$ and $a^{-1}aa^{-1}=a^{-1}$. An {\em inverse monoid\/} is an inverse semigroup with unit. We say that an inverse monoid $S$ {\em is generated} by a subset $\Lambda\subset S$ if $S$ coincides with the smallest subsemigroup of $S$  containing the set $\Lambda\cup\Lambda^{-1}$.

A {\em polycyclic monoid} is an inverse monoid $S$ with zero $0\ne 1$, which is generated by a subset $\Lambda\subset S$ such that $xx^{-1}=1$ for all $x\in \Lambda$ and $xy^{-1}=0$ for any distinct $x,y\in\Lambda$. If the generating set $\Lambda$ has cardinality $\lambda$, then $S$ is called a {\em $\lambda$-polycyclic monoid}. We claim that $|\Lambda|\ge 2$. In the opposite case, $\Lambda=\{x\}$ is a singleton and $0\in S=\{x^{-n}x^m:n,m\in\w\}$, which implies that $0=x^{-n}x^m$ for some non-negative numbers $n,m$. Then $0=x^{n+1}\cdot 0\cdot x^{-m}=x^{n+1}(x^{-n}x^{m})x^{-m}=x$ and hence $1=xx^{-1}=0x^{-1}=0$, but this contradicts the definition of a polycyclic monoid.

A canonical example of a $\lambda$-polycyclic monoid can be constructed as follows.
Let $\mathcal M_{\lambda^\pm}$ be the monoid of all words in the alphabet $\{x,x^{-1}:x\in\lambda\}$, endowed with the semigroup operation of concatenation of words. The empty word is the unit $1$ of the monoid $\mathcal M_{\lambda^\pm}$. Let $\mathcal M_{\lambda^\pm}^0:=\mathcal M_{\lambda^\pm}\cup\{0\}$ be the monoid $\mathcal M_{\lambda^\pm}$ with the attached external zero, i.e., an element $0\notin\mathcal M_{\lambda^\pm}$ such that $0\cdot x=0=x\cdot 0$ for all $x\in\mathcal M_{\lambda^\pm}^0$. On the monoid $\mathcal M_{\lambda^\pm}^0$ consider the smallest congruence $\sim$ containing the pairs $(xx^{-1},1)$ and $(xy^{-1},0)$ for all distinct elements $x,y\in\lambda$. Then the quotient semigroup $\mathcal M_{\lambda^\pm}^0/_\sim$ is the required canonical example of a $\lambda$-polycyclic monoid, which will be denoted by $\mathcal P_\lambda$ and called {\em the $\lambda$-polycyclic monoid}.

Algebraic properties of the $\lambda$-polycyclic monoid were deeply investigated in \cite{BardGut-2016(1)}. According to \cite[Theorem 2.5]{BardGut-2016(1)}, the semigroup $\mathcal P_\lambda$ is congruence-free, which implies that each $\lambda$-polycyclic monoid is algebraically isomorphic to $\mathcal P_\lambda$.

The aim of this paper is to describe Hausdorff locally compact topologies on $\mathcal P_\lambda$, compatible with the algebraic structure of the semigroup  $\mathcal P_\lambda$. A suitable compatibility condition is given by the notion of a semitopological semigroup.

A {\em semitopological semigroup} is a semigroup $S$ endowed with a Hausdorff topology making the binary operation $S\times S\to S$, $(x,y)\mapsto xy$, separately continuous. If this operation is jointly continuous, then $S$ is called a {\em topological semigroup}.

For a cardinal $\lambda\ge2$ by ${\mathcal P}^d_\lambda$ we shall denote the $\lambda$-polycyclic monoid $\mathcal P_\lambda$ endowed with the discrete topology and by ${\mathcal P}^c_\lambda$ the monoid $\mathcal P_\lambda$ endowed with the compact topology $\tau=\big\{U\subset\mathcal P_\lambda:0\in U\;\Rightarrow (\mathcal P_\lambda\setminus U\mbox{ is finite})\big\}$ of one-point compactification of the discrete space $\mathcal P_\lambda\setminus\{0\}$. It is clear that ${\mathcal P}^d_\lambda$ is a topological monoid. On the other hand, ${\mathcal P}^c_\lambda$ is a compact semitopological monoid, which is not a topological semigroup.

By \cite{BardGut-2016(1)}, each locally compact topological $\lambda$-polycyclic monoid is discrete and hence is topologically isomorphic to ${\mathcal P}^d_\lambda$. In the semitopological case we have the following dichotomy, which is the main result of this paper.

\begin{mtheo*} Any locally compact semitopological polycyclic monoid $S$ is either discrete or compact. More precisely, $S$ is topologically isomorphic either to ${\mathcal P}^d_\lambda$ or to ${\mathcal P}^c_\lambda$ for a unique cardinal $\lambda\ge 2$.
\end{mtheo*}

Since the compact semitopological $\lambda$-polycyclic monoid ${\mathcal P}^c_\lambda$ fails to be a topological semigroup, Main Theorem implies the mentioned result of \cite{BardGut-2016(1)}:

\begin{coro*} Any locally compact topological polycyclic monoid $S$ is discrete. More precisely, $S$ is topologically isomorphic to the topological $\lambda$-polycyclic monoid ${\mathcal P}^d_\lambda$ for a unique cardinal $\lambda\ge 2$.
\end{coro*}

Some other topologizability results of the same flavor can be found in \cite{Weil,Selden 1985,Hogan 1987,Banakh-Dimitrova-Gutik-2009,Mesyan-Mitchell-Morayne-Peresse-2013,Gutik-2015,Bardyla-2016,BardGut-2016(1),BardGut-2016(2)}.

\section*{Proof of Main Theorem}

The proof of Main Theorem is divided into a series of $12$ lemmas.

Let $S$ be a non-discrete locally compact semitopological polycyclic monoid and let $\Lambda$ be its generating set. By \cite[Proposition 2.2]{BardGut-2016(1)}, $S$ is algebraically isomorphic to the $\lambda$-polycyclic monoid $\mathcal P_\lambda$ for a unique cardinal $\lambda\ge 2$. So, we can identify $S$ with $\mathcal P_\lambda$ and the cardinal $\lambda$ with the generating set $\Lambda$ of the inverse monoid $S$.

Let $S^+$ be the submonoid of $S$, generated by the set $\Lambda$ (i.e., $S^+$ is the smallest submonoid of $S$ containing the generating set $\Lambda$). Elements of $S^+$ can be identified with words in the alphabet $\Lambda$. Such words will be called {\em positive}. The relations between the generators of $S$ guarantee that each non-zero element $a$ of $S$ can be uniquely written as $u^{-1}v$ for some positive words $u,v\in S^+$. Then by ${\downarrow}a$ we denote the set of all prefixes of the word $u^{-1}v$. For a subset $C\subset S$ we put ${\downarrow} C=\bigcup_{a\in C}{\downarrow} a$.

The following algebraic property of a polycyclic monoid is proved in \cite[Proposition~2.7]{BardGut-2016(1)}.

\begin{lemma}\label{proposition-2.8}
For any non-zero elements $a,b,c\in S$, the set $\{x\in S:axb=c\}$ is finite.
\end{lemma}

This lemma will be applied in the proof of the following useful fact, proved in \cite[Proposition~3.1]{BardGut-2016(1)}.

\begin{lemma}\label{l:1} All non-zero elements of $S$ are isolated points in the space $S$.
\end{lemma}

\begin{proof} For convenience of the reader we present a short proof of this important lemma.
First we show that the unit $1$ is an isolated point of the semitopological monoid $S$. Take any generator $g\in\Lambda$ and consider the idempotent $e=g^{-1}g$ of $S$. Since the map $S\to eS$, $x\mapsto ex$, is a retraction of the Hausdorff space $S$ onto $eS$, the principal right ideal $eS=g^{-1}S$ is closed in $S$. By the same reason, the principal left ideal $Se=Sg$ is closed in $S$. The separate continuity of the semigroup operation yields a neighborhood $U_1\subset S\setminus (g^{-1}S\cup Sg)$ of $1$ such that $0\notin (e\cdot U_1)\cap (U_1\cdot e)$. We claim that $U_1=\{1\}$. In the opposite case, $U_1$ contains some element $a\ne 1$, which can be written as $u^{-1}v$ for some positive words $u,v\in S^+$. Since $a\ne 1$ one of the words $u,v$ is not empty. If $u$ is not empty, then $a\in U_1\subset S\setminus g^{-1}S$ implies that the word $u^{-1}$ does not start with $g^{-1}$. In this case $ea=g^{-1}gu^{-1}v=g^{-1}\cdot 0=0$, which contradicts the choice of the neighborhood $U_1\ni a$. If the word $v$ is not empty, then $a\in U_1\subset S\setminus Sg$ implies that $v$ does not end with $g$. In this case $ae=u^{-1}vg^{-1}g=0$, again contradicting the choice of $U_1$.
This contradiction shows that the unit $1$ is an isolated point of $S$.

Now we can prove that each non-zero point $a\in S$ is isolated. Write $a$ as $u^{-1}v$ for some positive words $u,v\in S^+$. Since $uav^{-1}=1$, the separate continuity of the semigroup operation on $S$, yields an open neighborhood $O_a\subset S$ of $a$ such that $uO_av^{-1}\subset U_1=\{1\}$. By Lemma~\ref{proposition-2.8}, the neighborhood $O_a$ is finite and hence the singleton $\{a\}=O_a\setminus (O_a\setminus\{a\})$ is open, which means that the point $a$ is isolated in $S$.
\end{proof}

Lemma~\ref{l:1} implies that the locally compact space $S$ has a neighborhood base at zero, consisting of compact sets. It also implies the following useful lemma.

\begin{lemma}\label{l1} For any compact neighborhoods $U_0,V_0\subset S$ of zero  the set $U_0\setminus V_0$ is finite.
\end{lemma}

For an element $u\in S$ by $\mathcal R_u:=\{x\in S:xS=uS\}$ we denote its {\em Green $\mathcal R$-class} in $S$. Here $uS=\{us:s\in S\}$ is the right principal ideal generated by the element $u$.

\begin{lemma}\label{l:new} Every non-zero $\mathcal R$-class in $S$ coincides with the $\mathcal R$-class $\mathcal R_{u^{-1}}=\mathcal R_{u^{-1}u}$ for some positive word $u\in S^+$.
\end{lemma}

\begin{proof} Each non-zero element of the semigroup $\mathcal P_\lambda$ can be written as $u^{-1}v$ for some positive words $u,v\in S^+$. Taking into account that $u^{-1}v\cdot v^{-1}=u^{-1}$, we conclude that $\mathcal R_{u^{-1}v}=\mathcal R_{u^{-1}}=\mathcal R_{u^{-1}u}$.
\end{proof}

In the following Lemmas~\ref{l2}--\ref{main2} we assume that $U_0$ is any fixed compact neighborhood of zero in the semitopological monoid $S$. Since zero is a unique non-isolated point in $S$, the neighborhood $U_0$ is infinite.

\begin{lemma}\label{l2} The neighborhood $U_0$ has infinite intersection with some $\mathcal R$-class of $S$.
\end{lemma}

\begin{proof}
To derive a contradiction, assume $U_0$ has finite intersection with each $\mathcal R$-class of the semigroup $S$. Taking into account that $U_0$ is infinite and applying Lemma~\ref{l:new}, we can see that the set $B=\{u\in S^+: \mathscr{R}_{u^{-1}}\cap U_0\neq\emptyset\}$ is infinite. For every $u\in B$ denote by $v_{u}$ a longest positive word in $S^+$ such that $u^{-1}v_{u}\in \mathscr{R}_{u^{-1}}\cap U_0$ (such word $v_u$ exists as the set $\mathcal R_{u^{-1}}\cap U_0$ is finite). It follows that $A=\{u^{-1}v_{u}: u\in B\}$ is an infinite subset of $U_0$. Fix any element $g$ of the generating set $\Lambda$ of $S$. Since $0\cdot g=0$, we can use the separate continuity of the semigroup operation of $S$ and find a compact neighborhood $V_0\subseteq U_0$ of zero such that $V_0\cdot g\subseteq U_0$. But then $V_0\subseteq U_0\setminus A$ which contradicts Lemma \ref{l1}.
\end{proof}

\begin{lemma}\label{l3} The neighborhood $U_0$ has infinite intersection with each non-zero $\mathcal R$-class of the semigroup $S$.
\end{lemma}

\begin{proof}  By Lemma~\ref{l:new}, any non-zero $\mathcal R$-class of the semigroup $S=\mathcal P_\lambda$ is of the form $\mathcal R_{v^{-1}}$ for some positive word $v\in S^+$. By Lemmas~\ref{l:new} and \ref{l2}, for some element $u\in S^+$ the intersection $U_0\cap\mathcal R_{u^{-1}}$ is infinite. Observe that $v^{-1}u\cdot \mathcal R_{u^{-1}}\subset \mathcal R_{v^{-1}}$. By the separate continuity of the semigroup operation at $0=v^{-1}u\cdot 0$, there exists a neighborhood $V_0\subset S$ of zero such that $v^{-1}u\cdot V_0\subset U_0$. By Lemma~\ref{l1}, the difference $U_0\setminus V_0$ is finite, which implies that the intersection $V_0\cap \mathcal R_{u^{-1}}$ is infinite. Then the set $v^{-1}u\cdot (V_0\cap\mathcal R_{u^{-1}})\subset U_0\cap \mathcal R_{v^{-1}}$ is infinite, too.
\end{proof}

%Given two sets $A,B$, we write $A\subset^* B$ if $A\setminus B$ is finite.

%For a finite subset $F\subset \Lambda$, let $\mathcal P_F$ be the smallest subsemigroup of $S$ containing the set $F\cup F^{-1}\cup\{0,1\}$. %If $|F|>1$, then $\mathcal P_F$ is a finitely generated polycyclic monoid.

\begin{lemma}\label{l4} If the generating set $\Lambda$ is finite, then the neighborhood $U_0$ contains all but finitely many elements of the $\mathcal R$-class $\mathcal R_1=\{x\in S:xS=S\}$.
\end{lemma}

\begin{proof} To derive a contradiction, assume that the set $A:=\mathscr{R}_{1}\setminus U_0$ is infinite. We claim that for every $g\in \Lambda$ the set $A_{g}=\{a\in A: ag\in U_0\}$ is finite. Indeed, suppose that $A_{g}$ is  infinite. By Proposition \ref{proposition-2.8}, $A_{g}\cdot g$ is an infinite subset of $U_0$. Since $0\cdot g^{-1}=0$, the separate continuity of the semigroup operation on $S$ yields a compact neighborhood $V_0\subseteq U_0$ of zero such that $V_0\cdot g^{-1}\subseteq U_0$. Then $V_0\subseteq U_0\setminus (A_{g}\cdot g)$ which contradicts Lemma \ref{l1}.

Let $A^{*}=A\setminus \bigcup_{g\in \Lambda}{\downarrow}A_{g}$ (we recall that ${\downarrow}A_g=\bigcup_{a\in A_g}{\downarrow}a$ where ${\downarrow}a$ is the set of all prefixes of the word $a$). It follows that $A^{*}$ is a cofinite (and hence infinite) subset of $A$. Now we are going to show that $A^{*}$ is a right ideal of $\mathscr{R}_{1}$. In the opposite case we could find elements $c\in\mathscr{R}_{1}$ and $v\in A^{*}$ such that $vc\notin A^{*}$. Let $c^{*}$ be the longest prefix of $c$ such that $vc^{*}\in A^{*}$ (the word $c^{*}$ can be empty, in which case it is the unit of $S$). Then $vc^{*}g\notin A^{*}$ for some $g\in \Lambda$. Observe that $vc^*\in A^*\subset A\subset \mathcal R_1$ implies $vc^*g\in\mathcal R_1$. Assuming that $vc^*g\in U_0$, we conclude that $vc^*\in A_g\subset{\downarrow}A_g$, which contradicts the inclusion $vc^*\in A^*$. So, $vc^*g\notin U_0$ and hence $vc^*g\in A$. Then $vc^*g\notin A^*$ implies that $vc^*g\in{\downarrow}A_f$ for some $f\in \Lambda$ and thus $vc^*\in {\downarrow}A_f$, too. But this contradicts the inclusion $vc^*\in A^*$.
The obtained contradiction implies that $A^{*}$ is a right ideal of $\mathscr{R}_{1}$.

Let $u\in A^{*}$ be an arbitrary element. Since $u\cdot 0=0$, the separate continuity of the semigroup operation yields a compact neighborhood $V_0\subset U_0$ of zero such that $u\cdot V_0\subseteq U_0$. Proposition \ref{proposition-2.8} and Lemma~\ref{l3} imply that $u\cdot(V_0\cap\mathscr{R}_{1})$ is an infinite subset of $A^{*}\cap U_0\subset A\cap U_0$. In particular, $A\cap U_0$ is not empty, which contradicts the definition of the set $A:=\mathscr{R}_{1}\setminus U_0$.
\end{proof}

\begin{lemma}\label{l:8} If the cardinal $\lambda=|\Lambda|$ is finite, then the neighborhood $U_0$ contains all but finitely many elements of any $\mathcal R$-class $\mathcal R_x$, $x\in S$.
\end{lemma}

\begin{proof} The lemma is trivial if $x=0$. So we assume that $x\ne 0$. By Lemma~\ref{l:new}, $\mathcal R_x=\mathcal R_{u^{-1}}$ for some positive word $u\in S^+$.  Since $u^{-1}\cdot 0=0$, the separate continuity of the semigroup operation yields an neighborhood $V_0\subseteq U_0$ of zero such that $u^{-1}\cdot V_0\subseteq U_0$. By Lemmas~\ref{l1} and \ref{l4}, $\mathcal R_1\subset^* V_0$ (which means that $\mathcal R_1\setminus V_0$ is finite).
Then $\mathcal R_x=\mathcal R_{u^{-1}}=u^{-1}\cdot\mathscr{R}_{1}\subset^*u^{-1}\cdot V_0\subset U_0$, which means that $U_0$ contains all but finitely many points  of the $\mathcal R$-class $\mathcal R_x$.
\end{proof}

The following lemma proves Main Theorem in case of finite cardinal $\lambda=|\Lambda|$.

\begin{lemma}\label{main} If the cardinal $\lambda$ is finite, then the set $S\setminus U_0$ is  finite.
\end{lemma}

\begin{proof} To derive a contradiction, assume that $S\setminus U_0$ is infinite.  By Lemma~\ref{l:8}, for each $u\in S^+$ the set $\mathscr{R}_{u^{-1}}\setminus U_0$ is finite. Since the complement $S\setminus U_0=\bigcup_{u\in S^+}\mathcal R_{u^{-1}}\setminus U_0$ is infinite, the set $B=\{u\in S^+: \mathscr{R}_{u^{-1}}\setminus U_0\neq\emptyset\}$ is infinite, too. For every $u\in B$ denote by $v_{u}$ the longest word in $S^+$ such that $u^{-1}v_{u}\in \mathscr{R}_{u^{-1}}\setminus U_0$. Then $C=\{u^{-1}v_{u}: u\in B\}\subset \mathcal R_{u^{-1}}\setminus U_0$ is infinite and by Proposition \ref{proposition-2.8}, for every $g\in\Lambda$ the set $C\cdot g$ is an infinite subset of $U_0$. Since $0\cdot g^{-1}=0$, the separate continuity of the semigroup operation yields a neighborhood $V_0\subset U_0$ of zero such that $V_0\cdot g^{-1}\subseteq U_0$. By Lemma~\ref{l1}, the set $U_0\setminus V_0$ is finite. Since the set $Cg\subset U_0$ is infinite, there is an element $c\in C$ with $cg\in V_0$. Then $c=cgg^{-1}\in V_0g^{-1}\subset U_0$, which contradicts the inclusion $C\subset \mathcal R_1\setminus U_0$.
\end{proof}

\begin{lemma}\label{l5} The set $\mathscr{R}_{1}\setminus U_0$ is finite.
\end{lemma}

\begin{proof} To derive a contradiction, assume that the complement $A:=\mathscr{R}_{1}\setminus U_0$ is infinite. By Lemma~\ref{l3}, the set $U_0\cap\mathcal R_1$ is infinite.

 For a finite subset $F\subset \Lambda$, let $S_F$ be the smallest subsemigroup of $S$ containing the set $F\cup F^{-1}\cup\{0,1\}$. If $|F|\ge 2$, then $S_F$ is a polycyclic monoid.
Separately, we shall consider two cases.
\smallskip

1. First assume that for every finite subset $F\subset\Lambda$ the set $U_0\cap S_F$ is finite.  In this case for every point $g\in \Lambda$, consider the set $W_{g}=\{a\in U_0\cap\mathscr{R}_{1}: ag\notin U_0\}$. The separate continuity of the semigroup operation yields a neighborhood $V_0\subset U_0$ of zero such that $V_0\cdot g\subset U_0$. Lemma~\ref{l1} implies that the set $W_{g}\subset U_0\setminus V_0$ is finite and hence for every non-empty finite subset $F\subset\Lambda$ the set $U_{F}:=(U_0\cap\mathscr{R}_{1})\setminus\bigcup_{g\in F}W_g$ is infinite. We claim that $U_{F}\cdot y\subseteq U_{F}$ for every $y\in S_{F}\cap\mathscr{R}_{1}$. In the opposite case, there exist elements $y\in S_{F}\cap\mathscr{R}_{1}$ and $x\in U_{F}$ such that $xy\notin U_{F}$. Let $y^{*}$ be the longest prefix of $y$ such that $xy^{*}\in U_{F}$ (note that $y^{*}$ could be equal to $1$). Then $xy^{*}g\notin U_{F}$ for some $g\in F$. Hence $xy^{*}\in W_g$ which contradicts the definition of $U_{F}\ni xy^*$. Hence $U_{F}\cdot y\subseteq U_{F}$ for each element $y\in S_{F}\cap\mathscr{R}_{1}$.

Fix any element $v\in U_{F}$ and find a finite subset $D\subset \Lambda$ such that $v\in S_D$, $F\subset D$ and $|D|\ge 2$. Proposition \ref{proposition-2.8} implies that $v\cdot(S_{F}\cap\mathscr{R}_{1})$ is an infinite subset of $U_{F}\cap S_{D}$, which contradicts our assumption.
\smallskip

%Then since $\mathcal{P}_{D}$ is a closed and hence locally compact submonoid of $\mathcal{P}_{\lambda}$ which is isomorphic to the $\mathcal{P}_{m}$ for some positive integer $m$, Theorem \ref{main} implies that the set $\mathcal{P}_{D}\setminus U(0)$ is finite. Hence the set $U(0)\cap \mathcal{P}_{C}$ is infinite. The contradiction.

2. Next, assume that for some finite subset $F\subset\Lambda$ the intersection $U_0\cap S_{F}$ is infinite. For every $g\in F$ consider the subset $A_{g}:=\{a\in A: ag\in U_0\}$ of the infinite set $A=\mathcal R_1\setminus U_0$. The separate continuity of the semigroup operation yields a neighborhood $V_0\subset S$ of zero such that $V_0\cdot g^{-1}\subset U_0$. We claim that for every $a\in A_g$ we get $ag\notin V_0$. In the opposite case we would get $a=agg^{-1}\in V_0\cdot g^{-1}\subset U_0$, which contradicts the inclusion $a\in A$. Then $A_g=\{a\in A:ag\in U_0\setminus V_0\}$ and this set is finite by Lemmas~\ref{l1} and \ref{proposition-2.8}.  It follows that $A_F=A\setminus\bigcup_{g\in F}{\downarrow}A_g$ is a cofinite (and hence infinite) subset of $A$.

We claim that $A_{F}\cdot y\subseteq A_{F}$ for every $y\in S_{F}\cap\mathscr{R}_{1}$. In the opposite case, we can find elements $y\in S_{F}\cap\mathscr{R}_{1}$ and $x\in A_{F}$ such that $xy\notin A_{F}$. Let $y^{*}$ be the longest prefix of $y$ such that $xy^{*}\in A_{F}$ (note that $y^{*}$ could be equal to $1$). Then $xy^{*}g\notin A_{F}$ for some $g\in F$.
It follows from $xy^*\in A_F\subset A=\mathcal R_1\setminus U_0$ and $gg^{-1}=1$ that $xy^*g\in\mathcal R_1$. Assuming that $xy^*g\in U_0$, we conclude that $xy^*\in A_g$, which contradicts the inclusion $xy^*\in A_F$. So, $xy^*g\in\mathcal R_1\setminus  U_0=A$ and then $xy^*g\notin A_F$ implies that $xy^*g\in {\downarrow}A_h$ for some $h\in F$ and finally $xy^*\in{\downarrow}A_h$, which contradicts the inclusion $xy^*\in A_F$.
This contradiction completes the proof of the inclusion $A_{F}\cdot y\subseteq A_{F}$ for each $y\in S_{F}\cap\mathscr{R}_{1}$.

Fix any element $v\in A_{F}$ and find a finite subset $D\subset \Lambda$ such that $v\in S_D$, $F\subset D$ and $|D|\ge 2$. The subset $S_{D}$ contains the unique non-isolated point of the space $S$ and hence is closed in $S$. The local compactness of $S$ implies the local compactness of the polycyclic monoid $S_{D}$ endowed with the subspace topology. Lemma~\ref{l1} and our assumption guarantee that the semitopological polycyclic monoid $S_D$ is not discrete. By
Proposition \ref{proposition-2.8}, $v\cdot(S_{F}\cap\mathscr{R}_{1})$ is an infinite subset of $A_F\cap S_{D}\subset S_{D}\setminus U_0$. But this contradicts  Lemma~\ref{main} (applied to the locally compact polycyclic monoid $S_D$ and the neighborhood $U_0\cap S_D$ of zero in $S_D$).
\end{proof}

\begin{lemma}\label{c2} The neighborhood $U_0$ contains all but finitely many points of each $\mathcal R$-class in $S$.
\end{lemma}

\begin{proof} By Lemma~\ref{l:new}, it suffices to check that for any $u\in S^+$ the set $\mathcal R_{u^{-1}}\setminus U_0$ is finite. The separate continuity of the semigroup operation yields a compact neighborhood $V_0\subseteq U_0$ of zero such that $u^{-1}\cdot V_0\subseteq U_0$. By Lemmas~\ref{l5} and \ref{l1}, we get $\mathcal R_1\subset^* V_0$. Then $\mathcal R_{u^{-1}}=u^{-1}\cdot\mathcal R_1\subset^* u^{-1}\cdot V_0\subset U_0$, which means that the set $\mathscr{R}_{u^{-1}}\setminus U_0$ is finite.
\end{proof}

Our final lemma combined with Lemma~\ref{l:1} proves Main Theorem and shows that the semitopological polycyclic monoid $S$ carries the topology of one-point compactification of the discrete space $S\setminus\{0\}$.

\begin{lemma}\label{main2} The complement $S\setminus U_0$ is finite and hence $S$ is compact.
\end{lemma}

\begin{proof} To derive a contradiction, assume that the set $S\setminus U_0$ is infinite.
By Lemma~\ref{c2}, for each $u\in S^+$ the set $\mathcal R_{u^{-1}}\setminus U_0$ is finite. Since $S=\bigcup_{u\in S^+}\mathcal R_{u^{-1}}$, the set  $B=\{u\in S^+: \mathscr{R}_{u^{-1}}\setminus U_0\neq\emptyset\}$ is infinite. For every $u\in B$ denote by $v_{u}$ the longest word in $S^+$ such that $u^{-1}v_{u}\in\mathscr{R}_{u^{-1}}\setminus U_0$. Then $C=\{u^{-1}v_{u}: u\in B\}$ is an infinite subset of $S\setminus U_0$.  By Lemma \ref{proposition-2.8}, for any $g\in\Lambda$ the set $C\cdot g$ is infinite. The separate continuity of the semigroup operation yields a neighborhood $V_0\subset U_0$ of zero such that $V_0\cdot g^{-1}\subset U_0$. Then $V_0\subset U_0\setminus (C\cdot g)$ which contradicts Lemma~\ref{l1}.
\end{proof}

\section*{Acknowledgements}

The author acknowledges professors Taras Banakh and Oleg Gutik for their fruitful comments and suggestions.

%-------------------------------------------------------------------%

\end{document}